\documentclass[a4paper,11pt,reqno]{amsart}

\usepackage[utf8]{inputenc}
\usepackage[T1]{fontenc}
\usepackage{lmodern}
\usepackage{xfrac}
\usepackage[english]{babel}

\usepackage{tikz}

\definecolor{midnightblue}{rgb}{0.1, 0.1, 0.44}
\usepackage[colorlinks,allcolors=midnightblue]{hyperref}

\usepackage[a4paper,hmargin=2.4cm,top=2cm,bottom=3cm,footskip=3\baselineskip]{geometry}

\usepackage{fsmath,amsmath}
\usepackage{bm}
\newcommand{\loc}{\mathrm{loc}}
\DeclareMathOperator{\Ima}{Im}
\DeclareMathOperator{\sgn}{sign}
\newcommand{\R }{\mathbb R}
\newcommand{\N }{\mathbb N}
\newcommand{\C}{\mathbb C}
\newcommand{\Ll}{\mathcal L}
\newcommand{\chol}{\mathtt c}
\newcommand{\deltax}{{{\mathfrak h}}}

\begin{document}

\title[Continuous solutions to balance laws: H\"older regularity and applications]{H\"older regularity of continuous solutions to balance laws and applications in the Heisenberg group}

\author{L.~Caravenna} 
\author{E.~Marconi} 
\address{L.~Caravenna and E.~Marconi, Dipartimento di Matematica, Universit\`a degli Studi di Padova, Via Trieste 63, 35121 Padova, Italy}
\email{laura.caravenna@unipd.it, elio.marconi@unipd.it}

\author{A.~Pinamonti}
\address{Andrea Pinamonti, Department of Mathematics, University of Trento, Trento 38123, Italy}
\email{andrea.pinamonti@unitn.it}

%

\begin{abstract}
We prove H\"older regularity of any continuous solution $u$ to a $1d$ scalar balance law, when the source term is bounded and the flux is nonlinear of order $\ell\in \mathbb{N}$ with $\ell\ge 2$. Moreover, we prove that at almost every point $(t,x)$ it holds $u(t,x+h)-u(t,x)=o(|h|^{\frac1\ell})$ as $h\to0$.
Due to Lipschitz regularity along characteristics, this implies that at almost every point $(t,x)$, it holds $u(t+k,x+h)-u(t,x)=o((|h|+|k|)^{\frac{1}{\ell}})$ as $|(h,k)|\to0$.
We apply the results to provide a new proof of the Rademacher theorem for intrinsic Lipschitz functions in the first Heisenberg group.
\end{abstract}


\maketitle


\section{Introduction}
The effect of non-linearity in conservation laws is two-folds: on the one hand it prevents the existence of smooth solutions due to shocks formation. On the other hand it has some regularizing effects.
The most striking example is the estimate by Oleinik \cite{Ole}, which proves that $L^\infty$ initial data are immediately regularized to $BV_\loc$ at positive times, for entropy solutions of the classical Burgers' equation 
\[
u_t + [u^2/2]_x=0.
\]
Furthermore, \cite{ADL} proved that $BV_\loc$ regularity improves to $SBV_\loc$, then generalized to more general $1d$ scalar balance laws in \cite{Robyr} and to systems of conservation laws in \cite{BC}.
Similar results have been proven for Hamilton-Jacobi equations in several space dimensions starting from \cite{BDLR}.

A possible approach to prove the $BV_\loc$ estimate in the scalar case relies on the existence of a family non-crossing (backward) characteristics \cite[Chapter 10]{Daf_book} along which the solution is constant. The non-crossing condition is a geometric constraint which forces $BV_\loc$ regularity. 
This approach has been extended to the more general conservation law of the form:
\begin{align}\label{cl2}
u_t + [f(u)]_x=0,
\end{align}
see \cite{BGJ,AGV} for the case of convex fluxes $f$ and \cite{Daf, Cheng, Mar}, for the case of general smooth fluxes. The final result is that a fractional regularity of entropy solutions can be obtained quantifying the non-linearity of the flux $f$.

Motivated by some applications to the theory of rectifiable sets in the Heisenberg group, in this paper we consider the presence of the source $g\in L^{\infty}$ in \eqref{cl2}, i.e.
\begin{align}\label{eqn3}
u_t + [f(u)]_x=g,
\end{align}
We prove a regularity result for the continuous solutions of \eqref{eqn3}
where $f\in C^\ell(\mathbb{R})$ is nonlinear of order $\ell\ge 2$ (see Definition~\ref{nonlinearFInterval}). More precisely, we first demonstrate that every continuous solution to \eqref{eqn3} is automatically locally $1/\ell-$H\"older continuous and, even better, 
\begin{equation}\label{eqn2}
    u(t,x+h)-u(t,x)=o\left(\sqrt[\ell]{|h|}\right)\quad \mbox{as}\quad h\to0
\end{equation}
at Lebesgue almost every point $(t,x)$.
It is known that, even with analytic fluxes $f$ and continuous sources terms $g$, continuous solutions to \eqref{eqn3} improve to H\"older continuous regularity but do not go beyond. Indeed, even with quadratic flux continuous solutions to \eqref{E:base} might exhibit nasty fractal behaviour and in general they are neither Sobolev nor BV \cite{KSC, ABC3}. Moreover, if the flux $f$ is strictly convex but not analytic the result generally fails \cite[\S~5]{ABC3}, which motivates the finite order nonlinearity Assumption \ref{nonlinearFInterval} on the flux. Similar results have been obtained also for conservation laws in several space dimensions by means of the kinetic formulation (see for example \cite{LPT,GL}), nevertheless the conjectured optimal regularity has not been proven yet. 
We mention that also in the case of Hamilton-Jacobi equations, optimal H\"older regularity has been recently proved (see \cite{CV} and references therein).
\vspace{10pt}

In order to describe the application of our result to the theory of rectifiable sets in the Heisenberg group we recall that
the notion of Lipschitz submanifolds in sub-Riemannian geometry was introduced, at least in the setting of Carnot groups, by B. Franchi, R. Serapioni and F. Serra Cassano in a series of papers \cite{FSSC,FSSC2,FS} through the theory of intrinsic Lipschitz graphs (see also \cite{CMPSC1,CMPSC2,Didon}). Roughly speaking, a subset $S\subset \mathbb{G}$ of a Carnot group $\mathbb{G}$ is intrinsic Lipschitz if at each point $P\in S$ there is an intrinsic cone with vertex P and fixed opening, intersecting $S$ only in $P$. Remarkably, this notion turned out to be the right one in the setting of the intrinsic rectifiability in the simplest Carnot group, namely the Heisenberg group $\mathbb{H}^n$. Indeed, it was proved in \cite{FS} that the notion of rectifiable set in terms of an intrinsic regular hypersurfaces is equivalent to the one in terms of intrinsic Lipschitz graphs. 
We also remark that intrinsic Lipschitz functions played a crucial role in the recent paper \cite{NY}, where the longstanding question of determining the approximation ratio of the Goemans-Linial algorithm for the Sparsest Cut Problem was settled, see also \cite{NY2} for different applications.
We address the interested reader to \cite{FSSC2, Vit} for a complete introduction to the theory of intrinsic Lipschitz functions. One of the main open questions in this area of research is whether a Rademacher type theorem holds. 
Namely, assume that a splitting $\mathbb{G}=\mathbb{W}\mathbb{V}$ of a Carnot group $\mathbb{G}$ is fixed,
is it true that every intrinsically Lipschitz function $u:\mathbb{W}\to\mathbb{V}$ is intrinsically differentiable almost everywhere? In \cite{FMS,FSSC2}, it is proved that the answer is \emph{yes} if $\mathbb{G}$ is of step two or a group of type $\star$ and $\mathbb{V}=\mathbb{R}$. More recently, D. Vittone in \cite{Vit} proved that the answer is also \emph{yes} in the case of the Heisenberg group without any a priori assumption on the splitting. We also address the interested reader to \cite{AM,AM2,LDM} for further partial results. Remarkably, in \cite{JNGV} the authors constructed intrinsic Lipschitz graphs of codimension $2$ in Carnot groups which are not intrinsically differentiable almost everywhere. Insipired by the results contained in \cite{ASCV}, in \cite{BCSC}, the first named author together with Bigolin and Serra Cassano provided a characterization of intrinsic Lipschitz graphs in the Heisenberg group in terms of a system of non linear first order PDEs. Moreover, they proved the equivalence of different notions of continuous weak solutions to the equation 
\begin{align}\label{eqn}
u_t + [u^2/2]_x=g,
\end{align}
where $g$ is a bounded function. It turns out that the question whether an intrinsically Lipschitz function $u$ is intrinsically differentiable or is equivalent to \eqref{eqn2}.
This observation allows us to provide a completely PDEs based proof of the Rademacher theorem in $\mathbb{H}^1$.
We point out that the proof provided in \cite{FSSC2} requires some deep and nontrivial results in geometric measure theory, namely the fact that the subgraph of an intrinsic Lipschitz function $u$ is a set with locally finite $\mathbb{H}-$perimeter and that at almost every point of the graph of $u$ there is an approximate tangent plane. Another interesting feature of our approach is its potential applicability to those Carnot groups where the aforementioned geometric tools are not yet available e.g. the Engel group.\\[.3cm]

\textbf{Structure of the paper:}
In \S~\ref{S:Dafermos} we revise an elementary estimate on continuous solutions to balance laws when the flux is convex.
When the source is bounded, such estimate is the key to H\"older continuity and Lipschitz continuity along characteristics.
We prove in \S~\ref{S:Holder} H\"older regularity of any continuous solution $u$ to a $1d$ scalar balance law, when the source term is bounded and the flux is nonlinear of order $\ell$ in the sense of \S~\ref{S:nonlinearity}. The H\"older regularity is then refined in \S~\ref{S:oHolder}, where we show that at Lebesgue almost every point $(t,x)$ it holds $u(t,x+h)-u(t,x)=o(|h|^{\frac1\ell})$ as $h\to0$.
In \S~\ref{S:heisenberg} we finally prove the Rademacher theorem for intrinsic Lipschitz functions in $\mathbb{H}^1$.


\section{Preliminary results}
\label{S:Dafermos}

We recall the definition of characteristics of the equation
   \begin{equation}\label{E:base}
        \partial_t u + [f(u)]_x = g\,,
        \qquad f\in C^\ell(\setR), \ell\in\mathbb{N},\ g\ \text{bounded}.
    \end{equation}

\begin{definition}
A characteristic associated to a continuous solution $u$ to \eqref{E:base} is any  function $\gamma \in C^1(I)$, defined on a real interval $I \subset \mathbb{R}$ and satisfying the ordinary differential equation 
\[
\dot{\gamma}(t)=f'(u(s,\gamma(t))) \qquad \text{for }t \in I.
\]
\end{definition}

The following well-known result will be helpful later on (see \cite{Daf2,ABC1}):
\begin{proposition}\label{p:Dafermos}
    Let $g \in L^1_{\loc}(\Omega)$, $\Omega \subset \setR^2$ be open and assume that  $u \in C^0(\Omega)$ is a distributional solution of~\eqref{E:base}.
  Given a characteristic $\gamma:[t_1,t_2]\to \setR$ and $\deltax>0$, assume that the following sets
    \[
    \begin{split}
 S^+(\gamma,\deltax)\doteq&~ \{(t,x) \in [t_1,t_2]\times \setR : x \in [\gamma(t), \gamma(t) + \deltax]\}, \\ 
 S^-(\gamma,\deltax)\doteq&~ \{(t,x) \in [t_1,t_2]\times \setR : x \in [\gamma(t)-\deltax, \gamma(t)]\}
     \end{split}
    \]
    are contained in $\Omega$.
    If $f$ is convex on the interval $u(S^{+}(\gamma,\deltax)\cup S^{-}(\gamma,\deltax))$, then 
    \begin{subequations}\label{e:Dafermos_prel}
        \begin{align}\label{est1}
            \int_{\gamma(t_2)}^{\gamma(t_2) + \deltax} u(t_2, x) dx -   \int_{\gamma(t_1)}^{\gamma(t_1) + \deltax} u(t_1,x) dx 
            \le  \int_{S^+(\gamma,\deltax)} g \,dx dt \\\label{est2}
            \int_{\gamma(t_2)-\deltax}^{\gamma(t_2)} u(t_2, x) dx -   \int_{\gamma(t_1)-\deltax}^{\gamma(t_1)} u(t_1,x) dx 
            \ge  \int_{S^-(\gamma,\deltax)} g \,dx dt.
        \end{align}
    \end{subequations}  
    The analogous statement holds for concave fluxes $f$ reversing the inequalities~\eqref{est1} and \eqref{est2}.
\end{proposition}

\newcommand{\myCurve}[5]{%
	\draw[#3] (#1,0) .. controls (#1,1)  and (#1+0.4 ,#2-1) .. (#1+0.4,#2);
	\draw[#3, xshift=0.4cm, yshift=#2 cm] (#1,0) .. controls (#1,1)  and (#1+0.4 ,#2-1) .. (#1+0.4,#2) node [above, #5] {#4};
}%
\newcommand{\myCurveBold}[5]{%
	\draw[#3] (#1+.05,0.485) .. controls (#1+.06,1)  and (#1+0.4 ,#2-1) .. (#1+0.4,#2);
	\draw[#3, xshift=0.4cm, yshift=#2 cm] (#1,0) .. controls (#1,1)  and (#1+0.3 ,#2-1) .. (#1+0.35,#2-.485) node [above, #5] {#4};
}%
\newcommand{\myFill}[4]{%
\begin{scope}[xshift=#1 cm,yshift=#2 cm]%
	\fill[#3] (\myGamma,0) .. controls (\myGamma,1)  and (\myGamma+0.4 ,\mySigma-1) .. (\myGamma+0.4,\mySigma) -- (\myGamma-\myGammaCoef+0.4,\mySigma)--cycle;
	\fill[#3] (\myGamma,0)--(\myGamma-\myGammaCoef,0) .. controls (\myGamma-\myGammaCoef,1)  and (\myGamma-\myGammaCoef+0.4 ,\mySigma-1) .. (\myGamma-\myGammaCoef+0.4,\mySigma) -- cycle;
\draw (\myGamma/2+1.1,\mySigma/2) node {#4} ;
\end{scope}
}%
\begin{figure}
\centering
\begin{tikzpicture}[thick,scale=1.4]
\def\assexLung{5.5}%
\def\asseyLung{4.5}%
\def\mySigma{2}%
\def\myEpsilon{1.5}%
\def\myGamma{2.5}%
\def\myGammaCoef{.7}%
\def\tacca{2pt}
\draw[-stealth](-0.3,0)--(\assexLung+0.3,0) node[below]{$x$};
\draw[-stealth](0,-0.3)--(0,\asseyLung+0.3) node[left]{$t$};
\draw(0+\tacca,\mySigma)--(0-\tacca,\mySigma) node[left]{$\sigma$};
\draw(0+\tacca,\mySigma+\myEpsilon)--(0-\tacca,\mySigma+\myEpsilon) node[left]{\color{blue}$t_{2}=\sigma+\varepsilon$};
\draw(0+\tacca,\mySigma-\myEpsilon)--(0-\tacca,\mySigma-\myEpsilon) node[left]{\color{blue}$t_{1}=\sigma-\varepsilon$};;
\begin{scope}
	\clip (0,\mySigma-\myEpsilon) rectangle (5,5);
	\myFill{0}{0}{yellow}{}  
	\myFill{\myGammaCoef}{0}{orange}{}
\end{scope}
\begin{scope}
	\clip(0,0) rectangle (6,\mySigma+\myEpsilon);
	\myFill{0.4}{\mySigma}{yellow}{$S_{}^{-}$} 
	\myFill{0.4cm+\myGammaCoef}{\mySigma}{orange}{$S_{}^{+}$} 
\end{scope}
\myCurve{\myGamma}{\mySigma}{}{$\gamma(t)$}{} 
\myCurveBold{\myGamma+\myGammaCoef}{\mySigma}{very thick,blue}{$\gamma(t)+\varrho\varepsilon^\ell$}{above } 
\myCurveBold{\myGamma-\myGammaCoef}{\mySigma}{very thick,blue}{$\gamma(t)-\varrho\varepsilon^\ell$}{above} 
\draw[thin, dashed] (0,\mySigma)--(\myGamma+\myGammaCoef+0.4,\mySigma);
\draw[thin, dashed] (0,\mySigma+\myEpsilon)--(\myGamma+\myGammaCoef+0.75,\mySigma+\myEpsilon);
\draw[thin, very thick,blue] (\myGamma+\myGammaCoef-.65,\mySigma+\myEpsilon)--(\myGamma+\myGammaCoef+0.75,\mySigma+\myEpsilon);
\draw[thin, dashed] (0,\mySigma-\myEpsilon)--(\myGamma+\myGammaCoef+0,\mySigma-\myEpsilon);
\draw[thin, very thick,blue] (\myGamma-\myGammaCoef+.05,\mySigma-\myEpsilon)--(\myGamma+\myGammaCoef+.05,\mySigma-\myEpsilon);
\draw(4.5,2) node{\color{blue}$\mathbf{S_{\varrho}^{}}(\gamma,\varepsilon,\sigma)$};
\draw(1,1.5) node{\color{blue}$\mathbf{\deltax=\varrho\varepsilon^{\ell}}$};
\draw(1,1) node{\color{blue}$\mathbf{\varepsilon=\varrho^{-\frac1\ell}\sqrt[\ell]{{\deltax}{}}}$};
\filldraw[blue](\myGamma+0.4,\mySigma) circle (2pt) ;
\end{tikzpicture}
\caption{A region $S_{\varrho}$ of the Vitali covering~\eqref{e:defcover} and $S^{\pm}(\gamma|_{[t_{1},t_{2}]},\deltax)$ in Proposition~\ref{p:Dafermos}.}
\label{F:aree}
\end{figure}
\begin{remark}\label{r:Lip_char}
Dividing estimates \eqref{est1} and \eqref{est2} by $\deltax$ and letting $\deltax \to 0$ we get that $t \mapsto u(t,\gamma(t))$ is Lipschitz continuous with Lipschitz constant bounded by $\|g\|_{L^\infty(\Omega)}$. The same result was proven in \cite[Theorem~30]{ABC1} under the assumption that $\leb^1(\overline{\mathrm{Infl}(f)})=0$, where $\mathrm{Infl}(f)$ denotes the set of inflections points of $f$; it does not hold when inflection points are not negligible, see~\cite[Example~6.2]{ABC3}.
\end{remark}

\section{H\"older regularity}
\label{S:Holder}

The H\"older regularity of continuous solutions to~\eqref{eqn} in~\cite[Corollary~5.5]{BSC,BCSC} has been generalized in \cite{Car} to the case of $\ell$-nonlinear {convex} fluxes.
We describe in \S~\ref{S:nonlinearity} what we mean by $\ell$-nonlinear function: here we only mention that, in particular, are $\ell$-nonlinear the functions for which at any point there is a derivative of order between $2$ and $\ell$ which does not vanish.
In this section we extend the local $\frac1\ell$-H\"older regularity of continuous solutions without convexity hypothesis, so that it applies, for example, to $f(u)=u^{3}$.

\begin{proposition}\label{p:holder}
Let $f \in C^{\ell}(\mathbb{R})$ be nonlinear of order $\ell$ according to Definition~\ref{nonlinearFInterval}.
Consider an open connected set $\Omega \subset \setR^2$ and let $u \in C^0 (\Omega)$ and $g \in L^\infty(\Omega)$ be such that
\[
\partial_t u + [f(u)]_x = g \quad \mbox{in }\mathcal D'(\Omega).
\] 
Then $u \in C^{0,\alpha}_\loc (\Omega)$ with $\alpha=\frac1\ell$.
\end{proposition}

{Given $S$ any subset of $\mathbb{R}^n$, we denote by $B_r(S)$ the $r$-neighborhood of $S$, i.e.
\[
B_r(S):=\{x\in \mathbb{R}^n\ |\ d(x,S)\leq r\}.
\]
With this notation in hand we prove the following claim.}

\begin{lemma}\label{L:holder}
In the hypothesis of Proposition~\ref{p:holder}, consider $(t,x_1),(t,x_2) \in \Omega$ such that $f''$ does not vanish in the closed interval $I$ with endpoints $u(t,x_1)=u_{1}$ and $u(t,x_2)= u_{2}$: set 
\begin{align*}
&q=\frac{\min_{I} f''}{\max_{J} f''}\,,
&&G=\|g\|_{L^\infty(\Omega)}\,,
&&\tau= \frac{q}{2G}|u_{2}-u_{1}|\,,
&& C=\|f''\|_{L^\infty(B_{   G\tau}(I))}\,,
&&
L=\|f'\|_{L^\infty(B_{G\tau}(I))}\,,
\end{align*}
given any $J\supset B_{\frac12|u_{2}-u_{1}|}(I)$. 
If $\|g\|_{L^\infty(\Omega)}=0$ rather fix above any $G>0$.
Assume that
\begin{equation}\label{e:containedOmega}
    B_{{{ |x_1-x_2| +L\tau}}}(t,x_{1})\subset \Omega\,.
\end{equation}
Then, denoting by $c_{\ell}$ the constant of nonlinearity in Definition \ref{nonlinearFInterval}, we get
\[
|u_{2}-u_{1}| \le \chol |x_2-x_1|^{\frac1\ell}, \qquad \mbox{with} \ \chol= \sqrt[\ell]{\frac{ 4G}{q c_\ell}}.
\]
\end{lemma}

\begin{proof}[Proof of Lemma \ref{L:holder}]
We consider two characteristics $\gamma_1,\gamma_2$ passing through $(t,x_1), (t,x_2)$. Without loss of generality we assume $x_1<x_2$. 
We can also assume $f'(u_{1})>f'(u_{2})$, the argument in the opposite case is similar considering evolution in the past.


Since $\gamma_1$ is a characteristic, then
\begin{equation}\label{e:access1}
\gamma_1(t+\tau) = x_1 + \int_t^{t+\tau} f'(u(s,\gamma_1(s))) ds.
\end{equation}
By Remark~\ref{r:Lip_char} the Lipschitz function $u(\cdot, \gamma_1(\cdot))$ restricted to $[t, t + \tau]$ has image in $[u_{1} - G\tau , u_{1} + G \tau ]$.
By the chain rule on $\sigma\mapsto  f' (u (\sigma, \gamma_1(\sigma)))$ we estimate 
\begin{equation}\label{e:access2}
\begin{split}
 \left|f'(u(s,\gamma_1(s)))- f'(u_{1})  \right| = &\left| \int_t^s \partial_\sigma f' (u (\sigma, \gamma_1(\sigma))) d\sigma\right|
 \leq  C G (s-t)
 \qquad \forall s \in [t,t+\tau].
 \end{split}
\end{equation}
Estimating the integral in~\eqref{e:access1} {with $L$ and} by~\eqref{e:access2}, and doing similarly with $\gamma_{2}$, we see that
\begin{subequations}\label{E:curvestima}
\begin{align}
& \forall h\in[0,\tau] 
&& \abs{\gamma_1(t+h) - x_1 }\leq L h\,,
&&\gamma_1(t+h) \ge x_1 + f'(u_{1}) h - \frac{C G}2h^2\,,
\\
&&& \abs{\gamma_2(t+h) - x_2 }\leq Lh\,,
&&\gamma_2(t+h) \le x_2 + f'(u_{2}) h + \frac{C G}2h^2\,,
\end{align}
\end{subequations}
{ Since $\{(s,\gamma_1(s)),(s,\gamma_2(s)): s \in [t,t+\tau]\} \subset B_{L\tau}(t,x_1)\cup B_{L\tau}(t,x_2)$, by assumption~\eqref{e:containedOmega} then} $\{(s,\gamma_1(s)),(s,\gamma_2(s)): s \in [t,t+\tau]\} \subset \Omega${. By assumption on $\tau$  moreover} $2G\tau  \leq {|u_{2}-u_{1}|}$,  {so that} 
\[
(u_{2} - G\tau , u_{2} + G \tau ) \cap (u_{1} - G\tau , u_{1} + G\tau ) = \emptyset\,.
\]
In particular there is no $s \in [t, t + \tau)$ for which $\gamma_1(s) = \gamma_2(s)$, and therefore $\gamma_1(s)< \gamma_2(s)$ for every $ s \in [t, t + \tau)$.
Concatenating the previous inequalities~\eqref{E:curvestima}, as $ \gamma_1({{ t+\tau}})-\gamma_2({{ t+\tau}})\leq0$ we have
\[
f'(u_{1}) - f'(u_{2}) \le \frac{x_2-x_1}{\tau} +  C G  \tau.
\]
Lemma \ref{L:inversafprimo} allows to absorb $ C G \tau=\frac q 2\|f''\|_{L^\infty(B_{ \tau  G}(I))}|u_{2}-u_{1}|$ in the l.h.s., so that 
\[
\left( f'(u_1) - f'(u_2)\right) |u_1-u_2|q \le  {4 G } (x_2-x_1).
\]
By Lemma \ref{l:nondegenerate} we have $f'(u_1) - f'(u_2) \ge c_\ell|u_1-u_2|^{\ell -1}$,
thus we arrive to the claim 
\[
|u_1-u_2| \le  \left(\frac{{4G } }{ q c_\ell}\right)^\frac1\ell \sqrt[\ell]{x_2-x_1}\,.
\qedhere
\]
\end{proof}
\begin{example}\label{r:constant}
The solution $u(t,x)=\sgn x\cdot\sqrt{|x|}$ to  $\partial_{t}u+\partial_{x}u^{2}=\sgn x$ has H\"older constant $\sqrt2=\left(\frac{ 2\|g\|_{L^\infty}}{q c_\ell}\right)^\frac1\ell$.
\end{example}

\begin{proof}[Proof of Proposition~\ref{p:holder}]
Given any compact set $K\subset\Omega$ whose $t$-sections are intervals, set $J=u(K)$.
By Lemma~\ref{L:inversafprimo} the set $Z=\{v\in J\ :\ f''(v)=0\}$ is finite.
Let 
\[
0<\overline\delta<\frac13\min\{|v-w|\ :\ v,w\in Z\}\vee |J|
\]
be small enough so that the argument in limit in~\eqref{E:costanteNONdegere} of Lemma~\ref{lemma:nondegeneracy} is bigger than $2^{-\ell}<\frac{1}{2^{k-2}}$ for all $v\in Z$ when $\delta<\overline \delta$.
In particular, the constant $q$ in Lemma~\ref{L:holder} when $I\subseteq[2\delta,3\delta]$ and $J=[\delta,4\delta]$ when $\delta<\overline \delta$ is bigger than $2^{-\ell}$.

Consider now $x < y$, $(t,x),(t,y)\in K$. Denote by $[\![u_x,u_y]\!]$ the segment with endpoints $u(t,x)$, $u(t,y)$; we do not care of the order of endpoints in writing segments.
\begin{itemize}
\item[Case 1:] Suppose $[\![u_x,u_y]\!]\setminus B_{\overline\delta}(Z)\neq \emptyset$, thus it contains a segment $[\![u(t,x_1)- u(t,y_1)]\!]$ that do not intersect $B_{\overline\delta/16}(Z)$, for some $ x_1, x_2  \in [x,y]$, and $|u(t,x)- u(t,y) | \le \frac{8|J|}{\overline\delta}|u(t,x_1) - u(t,x_2)|$.
\item[Case 2:] Suppose $[\![u_x,u_y]\!]\subseteq [v-\overline\delta,v+\overline\delta]$ for some $v\in Z$. Let for example $u_y-v\geq|u_x-v|$, thus \[|u_y-u_x|/6\leq (|u_y-v|+|u_x-v|)/6\leq |u_y-v|/3\leq \overline \delta /3\,.\] Set $u_1=v+5|u_y-u_x|/6$ and $y_1=y$.
Let $x_1\in[x,y]$ such that $u_1=u(t,x_1)$.
\end{itemize}
In both cases, it is possible to find $ x_1, x_2  \in [x,y]$ such that $|u(t,x)- u(t,y) | \le \overline{ \chol}  |u(t,x_1) - u(t,x_2)|$, with $\overline{ \chol} = \frac{8|J|}{\overline\delta}\vee6$, and Lemma~\ref{L:holder} applies to $(t,x_1)$, $(t,x_2)$ yielding
\[
|u(t,x)- u(t,y) | \le \overline{ \chol} |u(t,x_1) - u(t,x_2)| \le \overline {\chol}  \chol |x_1-x_2|^{\frac1\ell} \le  \overline{ \chol}  \chol |y-x|^{\frac1\ell}
\,.
\]
By construction the {{ constant $\overline{ \chol}  \chol$}} is at most $\sqrt[\ell]{\frac{ 4G}{ c_\ell}}\cdot\left(2\vee\sqrt[\ell]{\frac{\min\{f''(x)\,:\,x\in J\setminus B_{\overline\delta/16}(Z)\}}{\max\{f''(x)\,:\,x\in J\}}}\right)$.

We thus proved that $u$ is $\frac1\ell-$H\"older in space on compact sets $K$ whose $t$-sections are intervals. In order to recover the regularity in time, recall that $u$ is $\|g\|_{L^\infty}-$Lipschitz along characteristics by Remark \ref{r:Lip_char}: as in at Step 2 in the proof of~\cite[Corollary 2.2]{Car} we get $\frac1\ell-$H\"older continuity on subsets of $\Omega$ of the form $\{(t,x)\in[t_{1},t_{2}]\times\R\ : \ \gamma_{1}(t)\leq x\leq\gamma_{2}(t)\}$, with $\gamma_{1}$, $\gamma_{2}$ characteristics.
Since the interior of such characteristic regions cover any compact subset of $\Omega$ we get the thesis.
\end{proof}

\begin{remark}\label{r:constants}
By looking at the proof of Proposition \ref{p:holder} we have that the $\ell$-H\"older constant of $u$ on a rectangle $K$ depends only on $f$ restricted to $u(K)$.
More precisely it depends on $c_\ell$ and on the minimum distance between zeros of $f''$.
\end{remark}

\section{Finer H\"older regularity}
\label{S:oHolder}

The main result of this section is the following:
\begin{proposition}\label{p:oHolder}
    Let $f \in C^2$ be nonlinear of order $\ell$ as in Definition~\ref{nonlinearFInterval}.
Let $\Omega \subset \setR^2$ and $u \in C^0 (\Omega)$, $g \in L^\infty(\Omega)$ be such that
\[
\partial_t u + [f(u)]_x = g \qquad \mbox{in }\mathcal D'(\Omega).
\] 
Then for $\leb^2-$a.e. $(t,x) \in \Omega$ it holds
\begin{equation}\label{e:oHolder}
    A(t,x)\doteq \limsup_{y \to x} \frac{|u(t,y)-u(t,x)|}{\sqrt[\ell]{|y-x|}}=0.
\end{equation}
\end{proposition}

By Remark \ref{r:constants}, the function $A$ is locally essentially bounded.

{
Before entering the core of the proof of Proposition~\ref{p:oHolder}, we prove a technical auxiliary lemma that is not related to the conservation law.

\begin{lemma}\label{L:technical}
    Let $u \in C^{0,\alpha}_{\loc}(\Omega)$, where $\alpha=\frac1\ell$. Define $A$ by~\eqref{e:oHolder}. Then for any $c\in\R$ it holds $A(t,x)=0$ at every $(t,x)$ which is a point of density one of the set $u^{-1}(\{ c\})$.
\end{lemma}
\begin{proof}
Consider any point $(t,x)$ of density one of the set $u^{-1}(\{ c\})$.
Up to a translation, fix $(t,x)=(0,0)$. Denoting $Q_r(\overline t,\overline x)=[\overline t-r,\overline t+r]\times[\overline x-r,\overline x-r]$, almost by definition
\begin{equation}\label{E:trivialDensity}
\lim_{r\downarrow 0}\frac{\Ll^2\left(Q_r(0,0)\setminus u^{-1}(c)\right)}{4 r^2}=0\,.
\end{equation}
Let $x_n \to 0$ with $x_n \ne 0$. 
Set $r_n=2|x_n|$.
For any $\varepsilon\in(0,1)$, being $Q_{\varepsilon r_n}(0,x_n)\subset Q_r(0,0)$ then
\[
\Ll^2\left(Q_{r_n}(0,0)\setminus u^{-1}(c)\right)
\geq 
\Ll^2\left(Q_{\varepsilon r_n}(0,x_n)\setminus u^{-1}(c)\right)
\]
converge to $0$, also divided by $r^2$, by~\eqref{E:trivialDensity}: then there are $(t_n,y_n) \in u^{-1}(\{c\})$ such that
\begin{equation}\label{e:densityone}
\lim_{n\to \infty} \frac{|t_n-t| + |y_n -x_n|}{|x_n-x|}=0.
\end{equation}
It follows from \eqref{e:densityone},  the equality $u(t,x)=u(t_n,y_n)=c$ and $u \in C^{0,\frac1{\ell}}_{\loc}(\Omega)$ that
\[
\begin{split}
\limsup_{x_n \to x}\frac{|u(t,x_n)-u(t,x)|}{\sqrt[\ell]{|x_n-x|}} \le &~ \limsup_{x_n \to x}\frac{|u(t,x_n)-u(t_n,y_n)|}{\sqrt[\ell]{|t_n-t| + |y_n -x_n|}} \frac{\sqrt[\ell]{|t_n-t| + |y_n -x_n|}}{\sqrt[\ell]{|x_n-x|}} \\
& + \limsup_{x_n \to x} \frac{|u(t_n,y_n)-u(t,x)|}{\sqrt[\ell]{|x_n-x|}}  \\
= &~ 0.
\end{split}
\]
Since $x_n\to x$ is arbitrary the lemma is proved.
\end{proof}
}

We will prove \eqref{e:oHolder} for points $(t,x)$ which are Lebesgue points of $g$ with respect to the following suitable coverings of $\Omega$, except for possibly an {{$\Ll^{2}$-negligible set of points contained in $\{(t,x)\in \Omega\ :\ f''(u(t,x))=0\}$}}. Such covering was first introduced in \cite{Car} for proving that for $\leb^2-$a.e.~$(t,x)$ the derivative of $\sigma\mapsto u(\sigma, \gamma(\sigma))$ along any characteristic $\gamma$ through $(t,x)$ is the Lebesgue value of $g$: for being self-contained, we prove this again in \S~\ref{S:heisenberg} (Case 2) for the quadratic flux; the same argument works for $\ell$-nonlinear fluxes by the estimates in this section.

Given $\varrho>0$, define regions translating a characteristic $\gamma$ of $\deltax=\varrho\varepsilon^{\ell}$, between times $\sigma\pm\varepsilon$:
\begin{equation}\label{e:defcover}
\begin{split}
\mathcal V_\varrho \doteq &~ \left\{ S_\varrho(\gamma, \varepsilon, \sigma) \ \big| \, \sigma \in \setR,\, \eps>0,\, \gamma\in C^{1}((\sigma - \eps, \sigma + \eps)\, : \, \dot\gamma(t)=f'(u(t,\gamma(t)))\right\} \,, \\
S_\varrho(\gamma,\eps,\sigma) \doteq &~ \left\{(t,x)\in[ \sigma-\eps, \sigma + \eps]\times\R \ :\  \gamma(t) - \varrho\eps^\ell \le x \le \gamma(t) + \varrho\eps^\ell \right\}.
\end{split}
\end{equation}

\begin{proposition}\label{p:Vitali}
    Consider any sequence $\{\varrho_{j}\}_{j\in\N}$ and $\mathcal  V_{\varrho_{j}}$ as in \eqref{e:defcover}.
    For every $q \in L^\infty(\Omega)$ a.e. $(t,x) \in \Omega$ is a Lebesgue point of $q$ with respect to the covering $\mathcal V_{\varrho_{j}}$, namely there is $O\subset \Omega$ with $\leb^2(\Omega \setminus O)=0$ such that for every $(t,x) \in O$ and every $j\in\N$ it holds
    \[
    \lim_{\substack{(t,x)\in S\in \mathcal  V_{\varrho_{j}}\\\diam(S)\downarrow0}}\frac{1}{\leb^{2}(S)}\int_{S}|q - q(t,x)| dx dt=0.
    \]  
\end{proposition}

The proof of the above proposition can be found in \cite[\S~3.1]{Car}: the key point is that one can define through the sets $S_\varrho(\gamma,\eps,\sigma)$ a quasi-distance for which the measure $\leb^2$ is doubling so that Lebesgue differentiation theorem applies.

We are now in position to prove Proposition \ref{p:oHolder}.

\begin{proof}[Proof of Proposition \ref{p:oHolder}]
Since $u \in C^{0,\alpha}_{\loc}(\Omega)$ by Proposition~\ref{p:holder}{, one can apply Lemma~\ref{L:technical}}: for any $c\in\R$ it holds $A(t,x)=0$ at every $(t,x)$ which is a point of density one of the set $u^{-1}(\{ c\})$.
Since $Z=\{v\ : \ f''(v)=0\}$ is discrete by Lemma~\ref{L:inversafprimo}, the statement thus holds for $\leb^2$-a.e. $(t,x) \in u^{-1}(Z)$.

We now prove the claim for $\leb^2$-a.e. $(t,x) \in \Omega \setminus u^{-1}(Z)$, namely such that $f''(u(t,x))\ne 0$.
Given two sequences $\varrho_j,\delta_j \downarrow 0$, $\varrho_j+\delta_j<1$, we consider the coverings $\mathcal  V_{\varrho_j}$ in~\eqref{e:defcover} and we set
\[
  r_j = \frac{\varrho_j^\ell}{1 + \delta_j}\quad\text{and denote\ } \varepsilon=\eps_j(\deltax)=  \frac{\sqrt[\ell]{(1+\delta_j)\deltax}}{\varrho_j}\text{ for }\deltax>0\,,
\]
so that $r_{j}\varepsilon^{\ell}=\deltax$ and $\varrho_{j}\varepsilon^{\ell}=\varrho_{j}^{-(\ell-1)}(1+\delta_j)\deltax$. 
We prove that $A(t,x)=0$ for every $(t,x) \in\Omega \setminus u^{-1}(Z)$ which is a Lebesgue point of $g$
with respect to each covering $\mathcal V_{\varrho_{j}}$ for $j \in \setN$.
Proposition~\ref{p:Vitali} grants that such points have full Lebesgue measure in $\Omega \setminus u^{-1}(Z)$.

Up to a translation we may assume that $(t,x)=(0,0)$.
{By the definition of $\limsup$ in~\eqref{e:oHolder}, one can extract a monotone sequence $x_k\to 0$, as $k\to+\infty$, such that
$A(0,0)= \lim_{k\to \infty} \frac{|u(0,x_k)-u(0,0)|}{\sqrt[\ell]{|x_k|}}$.
We claim that it is not restrictive to consider the case when also the following conditions hold: \[x_k \downarrow 0\,,\quad u(0,x_k) < u(0,0)\quad\text{ and }\quad f''(u(0,0))>0\,.\]
Indeed, this situation can be reached composing, if needed, the following transformations:
\begin{itemize}
    \item We already discussed $u^{-1}(Z)=u^{-1}(\{f''(u)=0\})$.
    \item If $f''(u(0,0))<0$, just reverse the time and consider $\overline u(t,x)=u(-t,x)$: $\overline u$ solves the balance law with flux $\overline f(v)=-f(v)$ and source $\overline g(t,x)=-g(-t,x)$ thanks to the fact that $u$ is a Kruzkov iso-entropy solution by~\cite[Lemma~42]{ABC1}.
    \item If $u(0,x_{k_j}) = u(0,0)$ for some  subsequence $k_j$, then $A(t,x)=0$ holds trivially and we stop here. If $u(0,x_k) - u(0,0)$ changes sign, extract a subsequence of $\{x_k\}_{k\in\N}$ along which $u(0,x_k) - u(0,0)$ does not change sign and it does not vanish. If $u(0,x_k) > u(0,0)$ moreover just consider $\widehat u(t,x)=-u(-t,x)$, that solves the balance law with source $\widehat g(t,x)=g(-t,x)$ and with flux $\widehat f(v)=f(-v)$ still satisfying $\widehat{ f}''(\widehat u(0,0))>0$.
    \item If $x_k\uparrow 0$, just reverse both time and space and consider $\widetilde u(t,x)=u(-t,-x)$: then $\widetilde u$ solves the conservation law with the same flux $\widetilde f(v)=f(v)$ and with source $\widetilde g(t,x)=-g(t,x)$.
\end{itemize}  }
Denote by $\gamma_0$ and $\gamma_k$ two characteristics such that $\gamma_0(0)=0$ and $\gamma_k(0)=x_k$. 
Set $u_{0}(t)=u(t,\gamma_{0}(t))$ and $u_{k}(t)=u(t,\gamma_{k}(t))$.
Let us fix $j \in \setN$ and denote by $S_k$ the set $S_{\varrho_j}(\gamma_0, \eps_j(x_k),0)$.
With the standard convention $\inf \emptyset = +\infty$, set 
\[
t_k^*= \eps_j(x_{k})\wedge \inf \left\{t>0 \ \ :\  u_{0}(t)\le u_{k}(t)\right\}\,.
\]
Notice that $t^*_k>0$, being $u$ continuous and $u_{k}(0)< u_{0}(0)$.
Since $x_k,t^*_k \to 0$ as $k\to \infty$ and $u$ is continuous, there is $\bar k$ such that for every $k>\bar k$ the function
$u$ restricted to $S_k$ takes values in $u^{-1}(\{f''>0\})$.
We will assume $k\ge \bar k$ in the following of the proof: thus 
 \[
 \dot\gamma_{0}(t)=f'(u_0(t))\ge f'(u_k(t))=\dot \gamma_{k}(t)
 \qquad \text{in $[0,t_k^*]$}\,
 \]
and $\gamma_{k}\leq \gamma_{0}+x_{k}$ by the initial condition, so that by the choice $\varrho_{j}\varepsilon^{\ell}> (1+\delta_j)\deltax$ we have
\begin{equation}\label{e:setinclusion}
\{ (t,x) \in [0,t^*_k]\times \setR: \gamma_0(t) - \delta_j x_k\leq x\leq \gamma_k(t)+\delta_j x_k\}\subset S_k\,.
\end{equation}

If $t \in [0,t_k^*]$, Proposition \ref{p:Dafermos} gives the one sided estimates
\begin{subequations}\label{e:Dafermos}
\begin{align}
\int_{\gamma_k(t)}^{\gamma_k(t)+\delta_j x_k} u(t,x)dx
-\int_{x_k}^{x_k+\delta_j x_k} u( 0,x)dx
\leq &
\int_{0}^{t}\int_{\gamma_k(t)}^{\gamma_k(t)+\delta_j x_k} g( s,x)dxds \,,
\\
\int^{\gamma_0(t)}_{\gamma_0(t)-\delta_j x_k} u(t,x)dx
-\int^{0}_{-\delta_j x_k} u( 0,x)dx
\geq &
\int_{0}^{t}\int_{\gamma_0(t)-\delta_j x_k}^{\gamma_0(t)} g(s,x)dxds \, .
\end{align}
\end{subequations}
Denote by $\chol$ the $\frac1\ell$-H\"older constant of $u$ restricted to $S_{\bar k}$.
Then when $\{t\}\times[x-r,x+r] \subset S_k$
\begin{equation}\label{e:errore_Holder}
\int_{x}^{x+r}\left|u(t,q)-u(t,x)\right|\,dq\leq \chol r^{\frac{\ell+1}{\ell}},\qquad \int_{x-r}^{x}\left|u(t,q)-u(t,x)\right|\,dq\leq  \chol r^{\frac{\ell+1}{\ell}}.
\end{equation}
By~\eqref{e:setinclusion} and by Proposition~\ref{p:Vitali} in particular for every $t\in [0,t^*_k]$ there is $\eta_k \downarrow 0$ such that
\begin{equation}\label{e:errore_Vitali}
    \left| \int_{0}^{t}\int_{\gamma_k(t)}^{\gamma_k(t)+\delta_j x_k} g( s,x)dxds - t \delta_j x_k g(0,0)  \right| \le \eta_k \leb^2(S_k)\,.
\end{equation}
Subtracting equations~\eqref{e:Dafermos} and dividing by $\delta_j x_k$, from~\eqref{e:setinclusion}, \eqref{e:errore_Holder} and \eqref{e:errore_Vitali} we deduce that
\begin{equation}\label{e:ineq_u_t}
    \begin{split}
        u_0(t)- u_k(t) & \ge u_{0}(0) - u_k(0) - 2\frac{\eta_k \leb^2(S_k)}{\delta_j x_k} - 2\chol \sqrt[\ell]{\delta_j x_k}
        \qquad
       \text{for every $t \in [0,t^*_k]$.} 
    \end{split}
\end{equation}

\noindent{\bf Case 1: there is $t \in [0,t^*_k]$ such that $u_0(t)-u_k(t) \le \frac12 (u_{0}(0) - u_k(0))$.}
From \eqref{e:ineq_u_t} it holds
\begin{equation}\label{e:case1}
\begin{split}\frac{u(0,0) - u(0,x_k)}{\sqrt[\ell]{x_k}} \le &~ 4 \frac{ \leb^2(S_k)}{\delta_j x_k \sqrt[\ell]{x_k}} \eta_k+ 4\chol \frac{\sqrt[\ell]{\delta_j x_k}}{\sqrt[\ell]{x_k}} \\
= &~ \frac{16 (1+\delta_j)^{\frac{\ell+1}{\ell}}}{\delta_j \varrho_j ^{\ell}}\eta_k + 4\chol \sqrt[\ell]{\delta_j}.
\end{split}
\end{equation} 

\noindent{\bf Case 2: for every $t \in [0,t^*_k]$ it holds $u_0(t)-u_k(t) > \frac12 (u_{0}(0) - u_k(0))$.}
In particular we have $t^*_k= \frac{\sqrt[\ell]{(1+\delta_j)x_k}}{\varrho_{j}}$ and $\gamma_{k}$, $\gamma_{0}$ do not intersect in $[0,t^*_k]$. By Lemma \ref{l:nondegenerate}, for every $t \in [0,t^*_k]$ it holds
\[
\begin{split}
\gamma_0'(t) - \gamma_k'(t) =  & ~ f'(u_0(t) - f'(u_k(t)) \\ 
\ge &~ 2c_\ell \left(u_0(t) - u_k(t)\right)^{\ell -1} \ge \frac{c_\ell}{2^{\ell -2}}(u_{0}(0) - u_k(0))^{\ell -1}
\end{split}
\]
Integrating with respect to time in $[0, t^*_k]$ and imposing that $\gamma_0(t^*_k)\le \gamma_k(t^*_k)$ we obtain
\[
0 \ge \gamma_0(t^*_k) - \gamma_k(t^*_k) \ge \gamma_0(0) - \gamma_k(0) + \frac{c_\ell}{2^{\ell -1}}(u_{0}(0) - u_k(0))^{\ell -1} t^*_k.
\]
Recalling that $\gamma_0(0)=0$ and $\gamma_k(0)=x_k$ this implies
\begin{equation}\label{e:case2}
\frac{u(0,0) - u (0, x_k)}{\sqrt[\ell]{x_k}} \le 2 \left(\frac{x_k}{t^*_k c_\ell}\right)^{\frac1{\ell -1}} \frac1{\sqrt[\ell]{x_k}} =2  \left( \frac{\varrho_j}{c_{\ell}\sqrt[\ell]{1+\delta_j}} \right)^{\frac1{\ell -1}}
\end{equation}

Taking first the limit as $k \to \infty$, where $\eta_k$ vanishes, \eqref{e:case1} and \eqref{e:case2} show that $A(0,0)=0$ because $\delta_j$ and $\varrho_j$ can be taken arbitrarily small.
\end{proof}

\section{A new proof of the Rademacher theorem for intrinsic Lipschitz functions}
\label{S:heisenberg}

Aim of this section is to provide a new proof of the Rademacher theorem for intrinsic Lipschitz functions in the first Heisenberg group $\mathbb{H}^1$. We start by recalling the main definitions, we refer the interested reader to \cite{SC} for a complete introduction to the subject.

We denote the points of $\mathbb{H}^1\equiv\mathbb{C}\times\R\equiv\R^{3}$
by
\[
  P=[z,t]=[x+iy,t]=(x,y,t),\qquad z\in\C,\ x,y\in\R,\ t\in\R.
\]
If $P=[z,t]$, $Q=[z',t']\in \mathbb{H}^1$ and $r>0$, the group
operation reads as
\begin{equation}
P\cdot Q:=\left [z+z',t+t'-\frac{1}{2}\Ima(\langle z,\bar{z'}\rangle)\right].
\end{equation}
The group identity is the origin $0$ and one has
$[z,t]^{-1}=[-z,-t]$. In $\mathbb H^1$ there is a natural one
parameter group of non isotropic dilations
$\delta_{r}(P):=[rz,r^2t]$, $r>0$.
The group $\mathbb{H}^1$ can be endowed with the homogeneous norm
\begin{equation*}
\| P\|_\infty:=\max\{|z|, |t|^{1/2}\}
\end{equation*}
and with the left-invariant and homogeneous distance
\begin{equation*}
   d_{\infty}(P,Q):=\| P^{-1}\cdot Q\|_\infty.
\end{equation*}
The metric $d_\infty$ is equivalent to the standard
Carnot-Carath\'eodory distance. It follows that the Hausdorff
dimension of $(\mathbb{H}^1,d_{\infty})$ is $4$, whereas
its topological dimension is clearly $3$.
The Lie algebra $\mathfrak{h}_1$ of left invariant vector fields
is (linearly) generated by
\begin{equation*}
  X =\partial_{x}-\frac{1}{2}y\partial_{t},
  \quad Y=\partial_{y}+\frac{1}{2}x\partial_{t},
  \quad T=\partial_{t}
\end{equation*}
and the only nonvanishing commutators are
\begin{equation*}
[X,Y] = T.
\end{equation*}

In the spirit of \cite{ASCV} we set 
$\mathbb W:=\{(x,y,t)\in\mathbb{H}^1:\ x=0\}\equiv\R^{2}$.
Therefore, if $A\in\mathbb W$, we write
$A=(y,t)$.

Following \cite{ASCV}, we define the graph quasidistance as:
\begin{definition}\label{fidistanzadef2708}
For $A=(y,t),\,B=(y',t')\in\omega$ we define
 $$d_{\phi}(A,B)=|y-y'|+\left|t'-t-\frac{1}{2}(\phi(A)+\phi(B))(y'-y)\right|^{1/2}.$$
\end{definition}
An intrinsic differentiable structure can be induced on $\mathbb{W}$ by means of $d_{\phi}$, see \cite{ASCV}.
We remind that a map $L:\mathbb{W}\to\mathbb{R}$ is $\mathbb{W}$-linear if it is a group homeomorphism and $L(ry,r^{2}t)=rL(y,t)$ for all $r>0$ and $(y,t)\in\mathbb{W}$. 
We recall then the notion of $\phi$-differentiablility.

\begin{definition}\label{defiWfdiff}
Let $\phi:\omega\subset\mathbb{W}\rightarrow\mathbb{R}$ be a fixed continuous function, and let
$A_0\in\omega$ and $\psi:\omega\rightarrow\mathbb{R}$ be given.  
We say that $\psi$ is $\phi$-differentiable at $A_0$ if there is an $\mathbb{W}$-linear functional $L:\mathbb{W}\rightarrow\mathbb{R}$ such that
\begin{equation}\label{definWfdiffernziabilita}
\lim_{A\rightarrow A_0}\frac{\psi(A) -\psi(A_0) - L(A_0^{-1}\cdot A)}{\d_{\phi}(A_0,A)} = 0.
\end{equation}
\end{definition}
It is well known \cite{ASCV} that
given a $\mathbb{W}$-linear functional $L:\mathbb{W}\rightarrow\mathbb{R}$ there exists a unique $\hat{w}\in\mathbb{R}$ such that
$L(A)=L((y,t))=\hat{w}y$.
Let us now introduce the concept of intrinsic Lipschitz function.
\begin{definition}
    Let $\phi:\omega\subset\mathbb{W}\to\mathbb{R}$. We say that $\phi$ is an intrinsic Lipschitz continuous function and we write $\phi\in Lip_{\mathbb{W}}(\omega)$ if there exists a constant $L>0$ such that
    \[
    |\phi(A)-\phi(B)|\leq L d_{\phi}(A,B)\quad \forall A,B\in\omega.
    \]
    We say that $\phi$ is locally intrinsic Lipschitz and we write $\phi\in Lip_{\mathbb{W},loc}(\omega)$ if $\phi\in Lip_{\mathbb{W},loc}(\omega')$ for any $\omega'\Subset\omega$.
\end{definition}
The following characterization is proved in \cite{BCSC}.
\begin{theorem}
    Let $\omega\subset \mathbb{R}^2$ be an open set and let $\phi:\omega\to\mathbb{R}$ be a continuous function. Then $\phi\in Lip_{\mathbb{W},loc}(\omega)$ if and only if there exists $g \in L^{\infty}(\omega)$ such that
    \[
    \phi_y+[\phi^2/2]_t=g\quad  \mbox{in}\ \mathcal{D}'(\omega).
    \]
\end{theorem}
Let us finally recall the following Rademacher type Theorem, proved in \cite{FSSC2} (see also \cite{Vit}). As recalled in the Introduction the original proof requires deep results in geometric measure theory, namely the fact that the subgraph of an intrinsic Lipschitz function $\phi$ is a set with locally finite $\mathbb{H}$-perimeter and that at almost every point of the graph of $\phi$ there is an approximate. Here instead we propose a completely PDEs based proof which uses the results of Section \ref{S:oHolder}.
\begin{theorem}
If $\phi\in Lip_{\mathbb{W}}(\omega)$ then $\phi$ is $\phi$-differentiable $\mathcal{L}^2-$a.e. in $\omega$.
\end{theorem}
More precisely, we prove $\phi$-differentiability at those points which are Lebesgue points of $g$ for all the coverings \eqref{e:defcover}, for any fixed sequence $\rho_i$ decreasing to 0.

\begin{proof}
{The proof relies on the fine estimates of the previous sections: we prefer within this proof to come back to the notations we had there for the variables: we call from now on $(t,x)$ the variables that at the beginning of this section were denoted as $(y,t)$.}

Let $A_0=(t,x)$ be a Lebesgue point of $g$ with respect to the coverings $\mathcal  V_{\varrho_{j}}$ of Lemma \ref{p:Vitali} and for simplicity we assume $A_0=(0,0)$.
We show that any sequence $A_n =(t_n,x_n) \to A_0$ has a (not relabeled) subsequence for which 
\[
\lim_{n \to +\infty} \frac{|\phi(A_n)-\phi(A_0)- g(A_0)t_n|}{d_\phi (A_n, A_0)} = 0.
\]
 Let $\gamma_n\in C^{1}([0,t_{n}])$ be a characteristic of $\phi$ through $(t_n,x_n)$ and set $A'_n=(0,\gamma_n(0))=(0,x'_n)$.
 {Since $A_n \to 0$ and $\gamma_n$ are uniformly Lipschitz, then $A'_n \to 0$. Without loss of generality we assume that $t_n, x'_n\in\left(0,\frac{1}{1+\|g\|_{L^\infty{(\omega)}}}\right)$.}
  Up to a (not relabeled) subsequence one of the following two cases is satisfied.

\vspace{2pt}

\paragraph{\textbf{Case 1}: It holds $t_n=o(\sqrt{x'_n})$ as $n\to +\infty$.}
In this case it follows from Proposition \ref{p:oHolder} and Remark \ref{r:Lip_char} that:
\begin{align}\label{ghq}
|\phi(A_n)-\phi(A_0)- g(A_0)t_n| &\leq\|g\|_{L^\infty(\omega)}t_n + |\phi(A_n)-\phi(A'_n)| + |\phi(A'_n)-\phi(A_0)| \\
\nonumber
&\leq 2\|g\|_{L^\infty(\omega)} t_n + o \left(\sqrt{x'_n}\right) = o \left(\sqrt{x'_n}\right)\,.
\end{align}
{
Let us write, \begin{align}\label{E:contotrinagolo2}
&d_{\phi}(A_n,A_0)=t_n+\sqrt{x_n'}+B(n).
\end{align}
where
\[B(n):=\left|x_{n}-\frac{\phi(A_n)+\phi(A_0)}{2}\cdot t_{n}\right|^{1/2}-\sqrt{x_n'}
\]
we claim that
\begin{equation}\label{limit}
B(n)=o(\sqrt{x_n'})\quad \mbox{as}\ n\to+\infty.
\end{equation}
First we observe that
\begin{align}
\bigg|x_{n}-\frac{\phi(A_n)+\phi(A_0)}{2}&\cdot t_{n}-x'_{n} \bigg| 
= \left|\int_{0}^{t_{n}}\phi(s,\gamma_n(s))\,ds -\frac{\phi(A_n)+\phi(A_0)}{2}\cdot t_{n}\right|\\
\nonumber
\leq &~ \left|\int_{0}^{t_{n}}\phi(s,\gamma_n(s))\,ds -\frac{\phi(A_n)+\phi(A_n')}{2}\cdot t_{n}\right| + \left|\frac{\phi(A_n')-\phi(A_0)}{2}\cdot t_{n} \right| \\
\nonumber
\leq &~\int_0^{t_n}\left|\frac{\phi(s,\gamma_n(s))-\phi(A_n)}{2}\right|+\left|\frac{\phi(s,\gamma_n(s))-\phi(A_n')}{2}\right|ds+\left |\frac{\phi(A_n')-\phi(A_0)}{2}\cdot t_{n} \right|\\
\nonumber
\leq &~
\|g\|_{L^\infty(\omega)}t_{n}^{2}+ o(\sqrt{x_n'})t_n,
\end{align}
where in the last inequality we used that $\phi$ is Lipschitz along the trajectory passing through $A_n$ and $A_n'$ for the first two terms and the little-H\"older regularity of $\phi$ in the $x$ variable for the third term.
Recalling that we are assuming $t_n=o(\sqrt{x_n'})$ as $n\to +\infty$, we deduce
\[
\bigg|x_{n}-\frac{\phi(A_n)+\phi(A_0)}{2}\cdot t_{n}-x'_{n} \bigg| = o(x'_n) \qquad \mbox{as }n \to +\infty,
\]
and since we assumed $x'_n>0$, then $x_{n}-\frac{\phi(A_n)+\phi(A_0)}{2}\cdot t_{n}>0$ for $n$ sufficiently large.
Therefore we can estimate
\[
|B(n)|^2 \le \bigg|x_{n}-\frac{\phi(A_n)+\phi(A_0)}{2}\cdot t_{n}-x'_{n} \bigg| = o(x'_n) \qquad \mbox{as }n \to +\infty
\]
and this proves the claim \eqref{limit}.
}

\vspace{2pt}

\paragraph{\textbf{Case 2}: there is $\delta\in (0,1)$ such that $\frac{t_n}{\sqrt{x'_n}} \ge \delta >0$ {for any $n\in\mathbb{N}$ sufficiently large.}}
Let $\varepsilon \in (0,1)$ and consider \[S_n^\varepsilon= \{ (s,x) \in [0,t_n] \times \setR: \gamma_n(s)-\varepsilon t_n^2 \le x \le \gamma_n(s)+\varepsilon t_n^2 \}.\]
Denote by $\chol$ the $\frac12-$H\"older constant of $\phi$ in a ball $B$ centered at $A_0$. For $n$ sufficiently large $A_n, A'_n \in B$: by estimates like~\eqref{e:Dafermos}-\eqref{e:errore_Holder} we get
\begin{equation}\label{e:DafinHeis}
|\phi(A_n)-\phi(A_0)- g(A_0)t_n| \le 4 \chol
 \sqrt{\varepsilon}t_n + \frac1{2\varepsilon t_n^2}\int_{S_n^\varepsilon} |g - g(A_0)|.
\end{equation}

Just for notational convenience, suppose $\varrho_0=1$. Since $S_n^\varepsilon \subset S_1\left(\gamma_n, t_n \sqrt{\frac1{\delta^2}+ \varepsilon}, 0\right) \in \mathcal V_1$ and $A_0$ is a Lebesgue point of $g$ with respect to the covering $\mathcal V_1$, then
\[
 \eta_n \doteq \frac{1}{\leb^2\left(S_1\left(\gamma_n, t_n \sqrt{\frac1{\delta^2}+ \varepsilon}, 0\right)\right)}\int_{S_n^\varepsilon}|g-g(A_0)|  
\]
vanishes as $n \to \infty$. 
The measure of the set is $4\left({\frac1{\delta^2}+ \varepsilon}\right)^{\frac32} t_{n}^{3}$.
It follows from \eqref{e:DafinHeis} that
\[\begin{split}
\frac{|\phi(A_n)-\phi(A_0)- g(A_0)t_n|}{t_n} \le & ~ 4\chol
 \sqrt{\varepsilon} + \frac{\leb^2\left(S_1\left(\gamma_n, t_n \sqrt{\frac1{\delta^2}+ \varepsilon}, 0\right)\right)}{2\varepsilon t_n^3} \eta_n \\
= &~ 4\chol
 \sqrt{\varepsilon} + \frac{2\left(\frac1{\delta^2}+ \varepsilon\right)^{\frac32}}{\varepsilon}\eta_n.
\end{split}
\]
Recalling that $t_n \leq d_\phi(A_n,A_0)$, the claim follows by letting $n\to \infty$ since $\varepsilon$ is arbitrarily small.
\end{proof}

\appendix

\section{Nonlinear functions of a given order}
\label{S:nonlinearity}

We specify the assumption of finite order nonlinearity, for brevity also called $\ell$-nonlinearity.

\begin{definition}\label{nonlinearFInterval}
Let $I\subset \setR$ be an interval.
A function $f$ differentiable at $ v$ is ``nonlinear of order $\ell >1$ with constant $c>0$ at $ v$'' if there exists $\delta>0$ such that for every $v+h\in I\cap B_{\delta}(v)$ one has
\[
\left|f(v+h)-f(v)-f'( v)h\right|\geq c\left|h\right|^{\ell}\,.
\]
If the inequality holds for all $v,v+h\in I$ we call $f$ ``nonlinear in $I$ of order $\ell $ with constant $c$''.
\end{definition}

Of course, if $f$ is nonlinear of order $\ell >1$ with constant $c>0$ at $ v$ then it is also nonlinear at $v$ of every order $\ell'\geq \ell$ with any constant $0<c'$ just because 
\[
c|h|^{\ell}=c'\cdot\left(\frac{c}{c'|h|^{\ell'-\ell}}\right)\cdot |h|^{\ell'}>c'|h|^{\ell'}\qquad
\text{for $|h|^{\ell'-\ell}\leq \frac{c}{c'}$.}
\]

\begin{remark}
The best order of nonlinearity at a point $v$ is the order $\ell$ of the first non-vanishing derivative at the point, higher than the first.
It can be proved just applying the definition of differentiability $\ell$-times recursively, and the constant is arbitrarily close to $ \frac{1}{\ell!}|f^{(\ell)}( v)|$.
\end{remark}

\begin{lemma}\label{lemma:nondegeneracy}
Suppose $f\in C^{\ell}(I)$. If $\sum_{j=2}^{\ell}\frac{1}{j!}\left|f^{(j)}\right|\geq c$ in the interval $I$; then $f$ is nonlinear at any $v\in I$ of order $\ell$ with constant $c'$, for any $0<c'<c$.

Let $k>1$ be the minimum exponent of nonlinearity of $f$ at $v$ with $c_{k}$ the supremum of the relative constants: then $f^{(k)}(v)=k!c_{k}$ and $f^{(j)}(v)=0$ for $j=2,\dots,k-1$. 

{Finally, if $f''(v)=0$ and $k$ is the minimum exponent of nonlinearity of $f$ at $v$,} then
\begin{equation}\label{E:costanteNONdegere}
\liminf_{\delta\downarrow0}\frac{\min\{f''(x)\ :\ 2\delta\leq |x-v|\leq 3\delta\}}{\max\{f''(x)\ :\ \delta\leq |x-v|\leq 4\delta\}}\geq \frac{1}{2^{k-2}}\,.
\end{equation}
\end{lemma}

\begin{proof}


Let $v, v+h\in I$. 
Suppose $f^{(j)}(v)=0$ for $j=2,\dots,k-1$ and $f^{(k)}(v)\neq0$ for some $2\leq k\leq \ell$.
Just by Taylor's expansion
\[
\left|f(v+h)-f(v)-f'(v)h\right|=\left|\sum_{j=2}^{k-1}\frac{1}{j!}f^{(j)}(v)h^{j}+\frac{1}{k!}f^{(k)}(\xi_h)h^{k}\right|=\left|\frac{f^{(k)}(\xi_h)}{k!} \right|\cdot |h|^k\,,
\]
for some $\xi_h$ between $v$ and $v+h$.
When $k=\ell$, by assumption $\frac{1}{\ell!}\left|f^{(\ell)}(v)\right|\geq c$ and thus $f$ is nonlinear at $v$ of order $\ell$ with constant arbitrarily close to $c$, thanks to the continuity of $f^{(\ell)}$.
If $2\leq k<\ell$ then by the same argument $f$ is nonlinear of order $k$ with some positive constant, and thus nonlinear of order $\ell$ for any positive constant.

Suppose now that $k$ is the minimum exponent of nonlinearity of $f$ at $v$ and denote by $c_{0}$ the supremum of the relative constants: by the first part of the statement in particular $f^{(j)}(v)=0$ for $j=2,\dots,k-1$ and $|f^{(k)}(v)|\leq c_{0}k!$, otherwise there would be a nonlinearity constant bigger than $c_0$. {We now prove $|f^{(k)}(v)|\geq c_{0}k!$. For $n\in\N$, by Definition~\ref{lemma:nondegeneracy} of nonlinear function there is $\delta_n$ decreasing to $0$ such that 
\[
\left|f(v+h)-f(v)-f'(v)h\right|\geq (c_{0}-\sfrac1n) |h|^{\ell} \qquad \forall   |h|\leq \delta_n\,.
\]
Let $ \xi_h$ between $v$ and $v+h$ such that $f(v+h)-f(v)-f'(v)h= \frac{f^{(k)}}{k!}(\xi_h)$: thus $|f^{(k)}(\xi_h)|\geq (c_{0}-\sfrac1n)k!$ when $|h|\leq \delta_n$. Since $\xi_h\to v$ as $n\to+\infty$, by continuity we conclude $|f^{(k)}(v)|\geq c_{0}k!$.}

Consider now $\overline x= v+{\overline h}$ attaining the minimum of $f''$ in the interval $[2\delta, 3\delta]$ for $\delta$ small, then
\[
|f''(\overline x)|=\left|f''(v+{\overline h})-f''(v)\right|=\left| \frac{f^{(k)}(\xi_{{\overline h}})}{(k-2)!}{\overline h}^{k-2}\right|
 \geq
\frac{\left| f^{(k)}(\xi_{{\overline h}})\right|}{(k-2)!}(2\delta)^{^{k-2}}
\,.
\]
Consider now {$\widehat x= v+{\widehat h}$ }attaining the maximum of $f''$ in the interval $[\delta, 4\delta]$ for $\delta$ small, then
\[
|f''(\widehat x)|=\left|f''\left(v+{\widehat h}\right)-f''(v)\right|=\left| \frac{f^{(k)}({\widehat h})}{(k-2)!}{\widehat h}^{k-2}\right|
\leq
\frac{\left| f^{(k)}(\xi_{{\widehat h}})\right|}{(k-2)!}(4\delta)^{^{k-2}}
\,.
\]
In particular, since $ f^{(k)}(\xi_{{\widehat h}})$ and $ f^{(k)}(\xi_{{\overline h}})$ converge to the same value $f^{(k)}(v)\neq0$ as $\delta\downarrow0$,
\begin{equation*}
\liminf_{\delta\downarrow0}\frac{\min\{f''(x)\ :\ 2\delta\leq |x-v|\leq 3\delta\}}{\max\{f''(x)\ :\ \delta\leq |x-v|\leq 4\delta\}}\geq \frac{1}{2^{k-2}}\,.
\end{equation*}
\end{proof}

Nonlinearity provides a lower bound on the increments also of the first derivative of $f$:
\begin{lemma}\label{l:nondegenerate}
Suppose $f$ is nonlinear in an interval $I$ of order $\ell$ with constant $c$. Then
  \begin{equation}\label{e:estimatefprime}
\left|f'(v) - f'(w)\right| \ge 2 c |v-w|^{\ell -1}
\qquad \forall v,w\in I\,.
    \end{equation}
\end{lemma}
\begin{proof}
Just by Definition~\eqref{nonlinearFInterval} of nonlinearity when $h=w-v$ we have
\[
\left|f'(w)h-f'(v)h\right|
=\left|f(v+h)-f(v)-f'(v)h+f(w-h)-f(w)+f'(w)h\right|\geq 2c|h|^{\ell}\,.\qedhere
\]
\end{proof}


\begin{lemma}\label{L:inversafprimo}
If $f\in C^{2}(J)$ is nonlinear of order $\ell$ and constant $c$ in a compact interval $J$, then:
\begin{enumerate}
\item When $\ell=2$ the second derivative of $f$ has lower bound $ 2c$ in $J$ and it holds
\begin{equation}\label{e:ration}
         {|f'(v)-f'(w)|}{} \ge q_{w,v}\cdot \|f''\|_{L^{\infty}(J)}\cdot  |v-w|  \,,
\end{equation}
for all $v,w\in  J$, with $q_{w,v}\doteq\frac{2c}{ \|f''\|_{L^{\infty}(J)}}$ belonging in $(0,1]$ and independent on $w,v$. 
\item When $\ell>2$ the set $Z=\{v \in J: f''(v)=0\}$ is finite and inequality \eqref{e:ration} holds whenever $[w,v]\cap Z=\emptyset$ with  $q_{w,v}\doteq\frac{\min_{[w,v]} |f''|}{\max_{J} |f''|}$ belonging in $(0,1]$.
 \end{enumerate}
\end{lemma}

\begin{proof}
\begin{enumerate}
\item When $\ell=2$ the statement follows immediately by~\eqref{e:estimatefprime}.
\item At any accumulation point $\overline v$ of $D$ necessarily $f^{(j)}(\bar v)=0$ for every $j\ge 2$, thus by Taylor's expansion we would have $f(\overline v+h)-f(\overline v)-f'(\overline v)h=o(h^{\ell})$: this would contradict the finite order of nonlinearity of $f$.
In particular, also when $\ell>2$ the set $D$ must be discrete.

If $[u,v]\cap K=\emptyset$ then $f$ is nonlinear of order $2$ in $[u,v]$ and $ {|f'(v)-f'(w)|}=|f''(\xi)||v-w|$ with $\xi\in(v,w)$.
Writing $|f''(\xi)|\geq  \min_{[u,v]} |f''|\cdot \frac{\max_{J} |f''|}{\max_{J} |f''|}$ one has the claim.
\qedhere
\end{enumerate}
\end{proof}

\section*{Acknowledgments:} The authors would like to thank G. Alberti, S. Bianchini, R. Serapioni, F. Serra Cassano and D. Vittone for useful conversations on the topics of the paper.
The authors are partially supported by the Gruppo Nazionale per l’Analisi Matematica, la Probabilit\`a e le loro Applicazioni (GNAMPA) of the Istituto Nazionale di Alta Matematica (INdAM).
E. M. is partially supported by the European Union's Horizon 2020 research and innovation program under the Marie Sklodowska-Curie grant No. 101025032. 
L. C. and E. M. are partially supported by the PRIN
2020 ”Nonlinear evolution PDEs, fluid dynamics and transport equations: theoretical foundations and applications” and PRIN PNR P2022XJ9SX of the European Union – Next Generation EU.
A. P. is partially supported by the PRIN 2017  "Gradient flows, Optimal Transport and Metric Measure Structures".

No datasets were generated or analysed during the current study.


\end{document}